\documentclass[12pt, reqno]{amsart}
\usepackage{amsmath, amsthm, amscd, amsfonts, amssymb, graphicx, color, mathtools, dsfont}
\usepackage[bookmarksnumbered, colorlinks, plainpages]{hyperref}
\usepackage[T1]{fontenc}
\usepackage{bbm}
\usepackage{indentfirst}
\usepackage{enumerate}

\newtheorem{theorem}{Theorem}[section]
\newtheorem{lemma}[theorem]{Lemma}

\newtheorem{corollary}[theorem]{Corollary}
\theoremstyle{definition}

\newtheorem{example}[theorem]{Example}

\theoremstyle{remark}

\numberwithin{equation}{section}

\newcommand{\MM}{\ensuremath{\mathcal{M}}}
\newcommand{\BB}{\ensuremath{\mathcal{B}}}
\newcommand{\BR}{\ensuremath{\mathtt{Bor}}}

\newcommand{\CC}{\ensuremath{\mathcal{C}}}
\newcommand{\LL}{\ensuremath{\mathcal{L}}}
\newcommand{\RR}{\ensuremath{\mathbb{R}}}
\newcommand{\NN}{\ensuremath{\mathbb{N}}}

\newcommand{\bk}[2]{\ensuremath{\left\langle #1,#2 \right\rangle}}

\begin{document}
\setcounter{page}{1}

\centerline{}

\centerline{}

\title[The e-property of asymptotically stable Markov semigroups]{The e-property of asymptotically stable Markov semigroups}

\author[Ryszard Kukulski, Hanna Wojew{\'o}dka-{\'S}ci\k{a}{\.z}ko]{Ryszard Kukulski$^1$ and Hanna Wojew{\'o}dka-{\'S}ci\k{a}{\.z}ko$^{1,2}$}

\address{$^{1}$ Institute of Theoretical and Applied Informatics, \newline Polish Academy
		of \hbox{Sciences},  \newline Ba{\l}tycka 5, 44-100 Gliwice, Poland}
\address{$^{2}$ Institute of Mathematics,  University of Silesia in Katowice, 
\newline Bankowa 14, 
		40-007 Katowice, Poland}
\email{\textcolor[rgb]{0.00,0.00,0.84}{ryszard.kukulski@gmail.com}}
\email{\textcolor[rgb]{0.00,0.00,0.84}{hanna.wojewodka@gmail.com}}


\subjclass[2020]{Primary 60J25; Secondary 37A30, 46N30, 46E27, 60B10.}

\keywords{Markov semigroup,  e-property, equicontinuity, asymptotic stability, stochastic continuity, bounded-Lipschitz distance}


\begin{abstract}
{The relations between asymptotic stability and the {e-property} of Markov semigroups,  
acting on measures defined on general (Polish) metric spaces, are studied.  
While usually much attention is paid to asymptotic stability (and the e-property has been for years verified only to establish it), it should be noted that the e-property itself is also important as it, e.g., ensures that numerical errors in simulations are negligible. 

Here, it is shown that any asymptotically stable Markov-Feller semigroup with an invariant measure such that the interior of its support is non-empty satisfies the eventual e-property.  
Moreover, we prove that any Markov-Feller semigroup, which is strongly stochastically continuous, and which possesses the eventual e-property, also has the e-property. We also present an example highlighting that strong stochastic continuity cannot be replaced by its weak counterpart, unless a~state space of a~process corresponding to a Markov semigroup is a compact metric space.}
\end{abstract} \maketitle

\section*{Introduction}

Markov operators and Markov semigroups acting on  
measures come naturally from Markov chains and continuous-time Markov processes, respectively. 
We focus here on those whose underlying state space is a general (Polish) metric space.

The research on existence and uniqueness of \emph{stationary distributions} of Markov processes, as well as convergence of the law of these processes to their unique stationaries that is independent of their initial distributions (\emph{asymptotic stability}) has been carried out for years all over the world. In order to prove the above mentioned (and other) ergodic properties, various techniques referring to \emph{equicontinuity} properties of families of Markov operators have been introduced (see, e.g., \cite{stettner,lasota_szarek_06, szarek_worm_11,worm_10,dawid_11,czapla_horbacz_14, wedrychowicz_18}, where ergodic properties of Markov chains are established, or  \cite{szarek_sleczka_urbanski_10, kapica_szarek_sleczka_11}, where asymptotic behavior of continuous-time Markov processes are studied). Initially, these type of methods have been developed for Markov processes evolving on compact metric spaces (cf. \cite{jamison}) or locally compact Hausdorff topological spaces (cf. \cite{meyn_tweedie_12}). Further, the so-called \emph{lower-bound technique} for equicontinuous families of Markov operators has been introduced to prove results for processes evolving on general (Polish) metric state spaces 
(see, e.g., \cite{lasota_szarek_06, szarek_13_lbt}; cf.  \cite{szarek_03,szarek_06}, where one the first
results concerning asymptotics of Markov operators evolving in Polish metric spaces are obtained). In the literature the concepts like the \emph{e-property} \cite{szarek_worm_11,  dawid_11, czapla_horbacz_14}, the \emph{eventual e-property} \cite{worm_10,czapla_18}, \emph{the Ces\'aro e-property} \cite{worm_10} or even \emph{uniform equicontinuity on balls}  \cite{hille_horbacz_szarek_16} have been considered. 
Recently, by utilizing a~\emph{Schur-like property} for measures, the authors of \cite{hille_new} have established a~rigorous connection between the concepts of equicontinuity for Markov operators acting on measures and their dual operators acting on functions.

Studying the e-property, the natural question arises: \emph{Does asymptotic stability immediately imply the e-property?}. The answer is negative and it is given in \cite{hille_szarek_ziem_17}. More precisely, the authors have provided a counterexample of a Markov-Feller operator which is asymptotically stable, but which does not possess the e-property. Simultaneously, they have proved that any asymptotically stable Markov-Feller operator with an invariant probability measure such that the interior of its support is non-empty satisfies the e-property. The result has been later  generalised and made tight in \cite{kukulski_wojewodka}.

One of the main theorems of this paper (Theorem \ref{thm:e-prop}) is an equivalent of \hbox{\cite[Theorem 2.3]{hille_szarek_ziem_17}} for semigroups of Markov operators, and so its proof is based on certain ideas derived from \cite{hille_szarek_ziem_17}. 
The difference is that here the Feller property, asymptotic stability, and non-emptiness of the interior of the support of the unique invariant probability measure are not sufficient to guarantee the e-property (counterexamples are provided in Section \ref{sec:ex}). Under these assumptions, we can only establish the \emph{eventual e-property}, that is, the e-property that holds from a~certain point in time, rather than over the entire time interval $t\in\mathbb{R}_+$ (cf. \cite{czapla_18}).   
To prove the e-property, an additional assumption of \emph{strong stochastic continuity} (defined as in \cite{ethier_kurtz}) has to be made. 
More precisely, in Theorem \ref{thm:eventual_vs_e},  we demonstrate that any Markov-Feller semigroup that is strongly stochastically continuous, and possesses the eventual e-property, also has the e-property. 
Interestingly, the assumption cannot be weakened. Example \ref{ex:3} in Section \ref{sec:ex} illustrates that stochastic continuity in its weaker form, that is, with pointwise convergence in the place of  convergence in the supremum norm, does not necessarily imply the desired assertion.

An immediate consequence of Theorems \ref{thm:e-prop} and \ref{thm:eventual_vs_e} is that an asymptotically stable Markov-Feller semigroup possesses the e-property if it is also strongly stochastically continuous and the interior of the support of its unique invariant probability measure is nonempty (Corollary \ref{thm_final}).

Asymptotic stability, especially if achieved with the exponential rate, is one of the most desired ergodic properties of Markov processes, but it is the e-property that, if additionally met, guarantees that certain numerical errors can be treated as negligible in simulations (cf. also \cite{czapla_horbacz_14}, where it is proven that an asymptotically stable Markov-Feller operator converges to its stationary uniformly, provided that it satisfies a form of equicontinuity condition). This  indicates why the theoretical result established here is important also from the point of view of applications.

It is also worth noting that stochastic continuity is not a highly restrictive requirement when working with Markov semigroups. It is necessary, for instance, to ensure joint measurability (as stated in \cite[Proposition 3.4.5]{worm_10}) or to uniquely characterize a Markov-Feller semigroup through its weak infinitesimal operators (see \cite[Theorem 2.5]{b:dynkin1}).

After having already published the initial version of this paper as a preprint and having submitted it to the journal, the article \cite{liu_liu_23} on a similar topic has appeared online, also as a preprint. The authors of that article present a~slightly more general result by assuming \emph{eventual continuity} of a given Markov semigroup instead of its asymptotic stability.  In contrast, here (assuming asymptotic stability) we guarantee the e-property of a  Markov semigroup in its strongest version (which we discuss in more detail in Section~\ref{sec:e-def} and \textnormal{\hyperref[appendix2]{Appendix II}}). 
In addition to this, we focus on providing justifications for the tightness of the results presented in this article. In particular, we give the negative answer to the fifth open question stated in \cite{liu_liu_23}. The authors of \cite{liu_liu_23} ask whether the equivalence between the e-property and the eventual e-property of Markov-Feller semigroups is still preserved if only stochastic continuity is assumed (in the place of strong stochastic continuity). 
As already mentioned above, Example \ref{ex:3} in Section \ref{sec:ex} proves that it is not true, unless the underlying state space is compact (cf. \textnormal{\hyperref[appendix]{Appendix I}}).

The paper is organized as follows. In Section \ref{sec:1} we collect the notation used throughout the article, as well as some definitions and facts (relating, among others, to Markov semigroups and their ergodic properties) that we refer to later in the paper. Section \ref{sec:main} is devoted to the presentation of the main results, and it is divided into three parts, in which we formulate theorems (Section \ref{sec:thms}), provide the above mentioned examples, pointing out that asymptotically stable Markov-Feller semigroups that are not strongly stochastically continuous do not necessarily enjoy the e-property (Section \ref{sec:ex}), and prove all the results (Section \ref{sec:proofs}). In \textnormal{\hyperref[appendix]{Appendix I}} we prove that whenever a~state space of a Markov semigroup is a compact metric space, the definitions of stochastic continuity and strong stochastic continuity of this semigroup are equivalent. Finally, in \textnormal{\hyperref[appendix2]{Appendix II}} we analyze different definitions
of the e-property, provide examples showing that these definitions are in general not equivalent and prove that under certain conditions (i.e. asymptotic stability and stochastic continuity) they can be used interchangeably.

\section{Preliminaries}\label{sec:1}
	Let $(S,\rho)$ be a Polish metric space, endowed with a Borel $\sigma$-algebra $\BR(S)$. We will write $B(x,r)$ for an open ball in $S$ centered at $x\in S$ and of radius $r>0$. The closure and the interior of a set $A \subset S$ shall be denoted by $\mathtt{Cl}(A)$ and $\mathtt{Int} (A)$, respectively, while for the boundry of $A$ we will write $\partial (A)$.

	Let $\BB_b(S)$ stands for the family of all real-valued, bounded and Borel measurable functions on $S$, equipped with the supremum norm $\|\cdot\|_{\infty}$, and let $\CC_b(S)$ and $\LL_b(S)$ denote the subfamilies of $\BB_b(S)$ consisting of all continuous and all Lipschitz continuous functions, respectively. 
\subsection{Spaces of Measures}
	Let us write $\MM(S)$ and $\MM_1(S)$ for the set of all finite Borel measures, defined on $(S,\BR(S))$, and its subset  consisting of all probability  measures, respectively.

For brevity, in what follows, we will write $\bk{f}{\mu}$ for the Lebesgue integral $\int_Sf(x)\,\mu(dx)$ of $f\in \BB_b(S)$ with respect to a signed Borel measure $\mu$ on $S$.

	We say that a sequence $\{\mu_n\}_{n \in \NN}\subset \mathcal{M}(S)$ 
	\emph{converges weakly} to $\mu \in \MM(S)$ as $n\to\infty$  
	($\text{w}\,\text{-}\lim_{n\to\infty}\mu_n = \mu$) if $\lim_{n \to 
	\infty}\bk{f}{\mu_n}= \bk{f}{\mu}$ for any $f\in \CC_b(S)$.  
	Citing the Portmanteau theorem (cf., e.g., \cite[Theorem 13.16]{klenke_13}), in the case of $\mu,\mu_1,\mu_2,\ldots\in\mathcal{M}_1(S)$, the weak convergence of $\{\mu_n\}_{n \in \NN}$ to $\mu$ is equivalent to each of the following conditions:
	\begin{itemize}
	\item $\lim_{n \to \infty}\bk{f}{\mu_n}= \bk{f}{\mu}$ for all $f\in \LL_b(S)$,
	\item $\limsup_{n\to\infty}\mu_n(F)\leq\mu(F)$ for all closed sets $F\subset S$,
	\item $\liminf_{n\to\infty}\mu_n(G)\geq\mu(G)$ for all open sets $G\subset S$,
	\item $\lim_{n\to\infty}\mu_n(A)=\mu(A)$ for any $A\in\BR(S)$ with $\mu(\partial (A))=0$.
	\end{itemize}

	The \emph{support} of any measure $\mu \in \MM(S)$ shall be defined, as usual, by
	\begin{align*}
	\mathrm{supp} \; \mu=\{x \in S: \mu(B(x,r))>0 \mbox{ for any } r>0\}.
	\end{align*}

\subsection{Markov Semigroups}

	An operator $P: \MM(S) \to \MM(S)$ is called a~\emph{Markov operator} if
		\begin{itemize}
			\item $( \lambda \mu_1 + \mu_2)P= \lambda \,\mu_1P+ \mu_2P$ 
			for any $\lambda \in\mathbb{R}_+$ and any $\mu_1, \mu_2 \in \MM(S)$,
			\item $ \mu P (S) = \mu(S)$  for every $\mu \in \MM(S)$.
		\end{itemize}
We say that a Markov operator $P$ is \emph{regular} if it possesses a \emph{dual operator}, that is, a~linear map (denoted by the same symbol) which acts on $\BB_b(S)$ and satisfies
		\begin{equation*}
		\bk{f}{\mu P}=\bk{Pf}{\mu} \quad \text{for any}\quad f\in 
		\BB_b(S)\quad\text{and any}\quad \mu\in\MM(S).
		\end{equation*}
	 In this work, we will focus on \emph{Markov-Feller operators}, that is, regular Markov operators whose duals leave the space $\CC_b(S)$ invariant. 
	 
	 It is worth noting that for a given \emph{transition probability function} 
	 \linebreak\hbox{$P:S\times \BR(S)\to[0,1]$}, that is, a~map for which 
	 $P(x,\cdot):\BR(S)\to[0,1]$ is a~probability measure for any fixed $x\in 
	 S$ and $P(\cdot,A):S\to[0,1]$ is a~Borel measurable function for any 
	 fixed $A\in\BR(S)$, one may define a~regular Markov operator $P:\MM(S)\to\MM(S)$, along 
	 with its dual operator $P:\BB_b(S)\to\BB_b(S)$, as follows:
	 \begin{align}\label{pi_P_U}
	 \begin{aligned}
	 &\mu P(A)=\left\langle P(\cdot,A),\mu\right\rangle
	 \quad\text{for every}\quad A\in\BR(S)\quad\text{and every}\quad\mu\in\MM(S),\\
	 &Pf(x)=\left\langle f,P(x,\cdot)\right\rangle
	 \quad\text{for every}\quad x\in S\quad \text{and every}\quad f\in\BB_b(S).
	 \end{aligned}
	 \end{align}
	 
A family $\{P(t)\}_{t\in\mathbb{R}_+}$ is 
called a \emph{regular Markov semigroup} if it consists of regular Markov 
operators \hbox{$P(t): \mathcal{M}(S) \to \mathcal{M}(S)$}, $t\in\mathbb{R}_+$, 
which form a~semigroup (under composition) with the identity transformation $P(0)$ 
as the unity element. 
A~semigroup $\{P(t)\}_{t\in\mathbb{R}_+}$ is said to be \emph{Markov-Feller} if 
$P(t)$ has the Feller property for all $t\in\mathbb{R}_+$. It is called 
\emph{stochastically continuous at zero}  
if $\lim_{t\to 0^+}P(t)f(x)=f(x)$ for all $x\in S$ and all $f\in \CC_b(S)$ (cf. \cite{worm_10, b:dynkin1, komorowski_walczuk}). Clearly, stochastic continuity at zero implies right-continuity at every $t_0\in\mathbb{R}_0$, 
but not left-continuity (see \cite[p. 93]{ziemlanska_21}).  
We will say that a semigroup $\{P(t)\}_{t\in\mathbb{R}_+}$ is 
\emph{stochastically continuous} (left and right) if  $\lim_{t\to t_0}P(t)f(x)=P(t_0)f(x)$ 
for all $x\in S$, all $t_0 \in\mathbb{R}_+$ and all $f\in \CC_b(S)$.  If, in 
turn, $\lim_{t\to 0^+}\|P(t)f-f\|_{\infty}=0$ for every $f\in \CC_b(S)$, we say, 
after  \cite{ethier_kurtz}, that $\{P(t)\}_{t\in\mathbb{R}_+}$ is \emph{strongly 
continuous} (or, more precisely, \emph{strongly stochastically continuous}). 
From the above definitions one can easily deduce that strong stochastic continuity implies stochastic continuity (left and right) (cf. \cite{kantorovitz}), which in turn, obviously, gives stochastic continuity at zero.  
A reader interested in a~precise distinction between these notions is referred 
to Example \ref{ex:3}, where we define a  stochastically continuous but not 
strongly stochastically continuous Markov semigroup, and to  
\textnormal{\hyperref[appendix]{Appendix I}}, where we demonstrate that in the case 
of a~compact metric state space all these notions coincide.

In the end, let us indicate that Markov semigroups are intimately connected with 
Markov processes. Given a Markov semigroup $\{P(t)\}_{t\in\mathbb{R}_+}$ and some measure 
$\mu\in\mathcal{M}_1(S)$, we can define a time-homogeneous \emph{Markov process} $\Phi$ with transition semigroup $\{P(t)\}_{t\in\mathbb{R}_+}$ and initial distribution $\mu$  as a family of $S$-valued random variables 
\hbox{$\{\Phi(t)\}_{t\in\mathbb{R}_+}$}, 
on some probability space $(\Omega,\mathcal{F},\mathbb{P}_{\mu})$,  satisfying 
\begin{align}\label{eq:M}
\begin{aligned}
&\mathbb{P}_{\mu}\left(\Phi(0)\in A\right)=\mu(A),\\
&\mathbb{P}_{\mu}\left(\Phi(s+t)\in A \,\lvert\, 
\mathcal{F}(s)\right)=\mathbb{P}_{\mu}\left(\Phi(s+t)\in A \,\lvert\, 
\Phi(s)\right)\quad\text{a.s.},\\
&\mathbb{P}_{\mu}\left(\Phi(s+t)\in A \,\lvert\, 
\Phi(s)=x\right)=P(t)(x,A)
\end{aligned}
\end{align}
for all $A \in\BR(S)$, all $s, t\in\mathbb{R}_+$ and any $x\in S$,
where $\{\mathcal{F}(s)\}_{s\in\mathbb{R}_+}$ is the natural filtration of $\Phi$.

\subsection{Ergodic Properties of Markov Semigroups}\label{sec:e-def}

Let us now recall the most important  notions concerning the ergodicity of Markov semigroups, including asymptotic stability, the eventual e-property and the e-property.

A Markov semigroup $\{P(t)\}_{t\in \mathbb{R}_+}$ is said to be 
\emph{asymptotically stable} if there exists a~unique probability measure \hbox{$\mu_*$} such that $ \mu_*P(t) = \mu_*$ for any \hbox{$t\in \mathbb{R}_+$} 
($\mu_*$ 
is then called an \emph{invariant measure} of $\{P(t)\}_{t\in\mathbb{R}_+}$) and 
\hbox{$\text{w}\,\text{-}\lim_{t\to\infty}\mu P(t) = \mu_*$} for each $\mu 
\in \MM_1(S)$. 	 

Let $R$ be a sufficiently large subset of the set $\CC_b(S)$.  
According to \cite{worm_10}, a~Markov semigroup 
$\{P(t)\}_{t\in\mathbb{R}_+}$ has the \emph{eventual e-property in $R$ at $z \in S$} if there exists $\tau\in\mathbb{R}_+$ such that for any $f \in R$
	 \begin{equation}\label{e-prop}
	 \lim_{x \to z}\sup_{t \geq\tau} \,\lvert P(t)f(x)-P(t)f(z)\rvert=0,
	 \end{equation}
	and it has the \emph{e-property in $R$ at $z$} as long as the above 
	holds with the supremum taken over the entire set $\mathbb{R}_+$. 
	 If \eqref{e-prop} holds for each $z \in S$, then we say that 
	 $\{P(t)\}_{t\in\mathbb{R}_+}$ has the \emph{(eventual) \hbox{e-property} in $R$}. A more general definition of the \emph{eventual 
	 e-property in $R$ at $z\in S$} has been introduced in 
	 \cite{czapla_18}, where the author, instead of \eqref{e-prop}, only requires 
	 the following:
\begin{equation*}
\forall_{f \in R} 
\;\forall_{\varepsilon>0}\;\exists_{\delta>0}\;\exists_{\tau\in\mathbb{R}_+}\;\forall_{x\in
 B(z,\delta)}\;\forall_{t\geq\tau}\;\;
\lvert P(t)f(x)-P(t)f(z) \rvert < \varepsilon.
\end{equation*} 
In what follows, we will adhere to the latter (more general) definition.

In the end, let us indicate that in many papers (such as  \cite{worm_10,hille_szarek_ziem_17,czapla_18, ziemlanska_21, liu_liu_23}, just to name a few), $R$ is assumed to be the set $\LL_b(S)$, although it can also be the set of bounded continuous functions with bounded (or compact) support (cf., e.g., \cite{stettner,meyn_tweedie_12, dawid_11}), or even the entire set $\mathrm{C}_b(S)$ (as in \cite{kukulski_wojewodka} or in this paper). Definitions of the e-property in different sets $R$ are generally non-equivalent. We elaborate more on this in \textnormal{\hyperref[appendix2]{Appendix II}}, where we provide some examples proving this statement. Moreover, we show that the notions of the e-property in some of the sets $R$ can be used interchangeably under the assumption that a~given Markov semigroup is asymptotically stable and stochastically continuous at zero. Similar observation is derived in the case of a single Markov operator in \cite[Remark 2.1 and Lemma 3.4]{kukulski_wojewodka} (although there only asymptotic stability is required).

\section{Main Results}\label{sec:main}

Within this section we will formulate and establish the main results of this article. The theorems are supported by examples that demonstrate the necessity of the undertaken assumptions.

\subsection{Theorems}\label{sec:thms}
The main results of this paper read as follows:
\begin{theorem}\label{thm:e-prop}
Let $\{P(t)\}_{t\in\RR_+}$ be an asymptotically stable Markov-Feller semigroup, and let $\mu_*$ be its unique invariant probability measure. 
If
\begin{align}\label{notempty}
\mathtt{Int}\left(\mathrm{supp}\left(\mu_*\right)\right)\neq\emptyset,
\end{align}
then $\{P(t)\}_{t\in\RR_+}$ has the eventual e-property in $\CC_b(S)$.
\end{theorem}

\begin{theorem}\label{thm:eventual_vs_e}
Let $\{P(t)\}_{t\in\RR_+}$ be a strongly stochastically continuous Markov-Feller semigroup. If $\{P(t)\}_{t\in\RR_+}$ has the eventual e-property in $\CC_b(S)$, then it has the \hbox{e-property} in $\CC_b(S)$ too.
\end{theorem}

To establish Theorem \ref{thm:eventual_vs_e}, we employ the following lemma, which 
in itself is an interesting result.

\begin{lemma}\label{thm:e-prop_[0,T]}
Let $\{P(t)\}_{t\in\RR_+}$ be a strongly stochastically continuous Markov-Feller 
semigroup. Then, for every $T\in\mathbb{R}_+$, $\{P(t)\}_{t\in \RR_+}$ has the 
e-property in $\CC_b(S)$ on the time interval $[0, T]$, i.e.
\begin{align}
	 \lim_{x \to z} \sup_{t \in [0, T]} \,\lvert P(t)f(x)-P(t)f(z)\rvert = 0
	 \quad\text{for every}\quad z\in S\quad \text{and every}\quad f\in\CC_b(S).
\end{align}
\end{lemma}

Finally, as an immediate consequence of Theorems \ref{thm:e-prop} and \ref{thm:eventual_vs_e}, we get the following statement:
\begin{corollary}\label{thm_final}
Let $\{P(t)\}_{t\in\RR_+}$ be a strongly stochastically continuous and asymptotically stable Markov-Feller semigroup, and let $\mu_*$ be its unique invariant probability measure satisfying \eqref{notempty}. 
Then $\{P(t)\}_{t\in\RR_+}$ has the e-property in $\CC_b(S)$.
\end{corollary}

\subsection{Examples}\label{sec:ex}
Before we proceed to the proofs of the above-stated results, let us first present two  examples. 

Example \ref{ex:1} demonstrates that condition \eqref{notempty}, imposed in 
\cite[Theorem 2.3]{hille_szarek_ziem_17}  on  an asymptotically stable 
Markov-Feller chain to show that is has the \hbox{e-property}, is insufficient for a similar result in the case of a Markov semigroup. More precisely, Example \ref{ex:1} illustrates that, 
without additional assumptions, beyond those stated in Theorem \ref{thm:e-prop}, it 
is not possible to establish the e-property. Instead, only the eventual e-property 
can be concluded.

Example \ref{ex:3}, in turn, highlights the necessity of requiring strong stochastic continuity, rather than just stochastic continuity, to guarantee that an asymptotically stable Markov-Feller semigroup 
with unique invariant probability measure satisfying \eqref{notempty} has the e-property.

\begin{example}\label{ex:1}
Let $S=\{0\}\cup\{1/n:\,n\in\NN\}$, and note that $S$, endowed with the Euclidean metric, is a Polish space. 
For every $n\in\NN$, every $t\in\mathbb{R}_+$ and any $A\in\BR(S)$ define
\begin{align}\label{ex:def:pi}
P(t)\left(1/n,A\right)=e^{-nt}\delta_{1/n}(A)+\left(1-e^{-nt}\right)\delta_1(A),
\end{align}
and let
\begin{align}\label{ex:def:pi(0)}
\begin{aligned}
P(t)(0,A)
&=\mathbbm{1}_{\{0\}}(t)
\delta_0(A)
+\left( 1-\mathbbm{1}_{\{0\}}(t) \right) \delta_1(A).
\end{aligned}
\end{align}

It is not hard to see that, for any  $t\in\RR_+$, $P(t):S\times \BR(A)\to[0,1]$, given by \eqref{ex:def:pi} and \eqref{ex:def:pi(0)}, is a transition probability function. Thus, we can define the corresponding Markov operator $P(t)$  as indicated in \eqref{pi_P_U}. Obviously,  $\{P(t)\}_{t\in\mathbb{R}_+}$ is then a~Markov semigroup. To establish the Feller property it suffices to observe that for all $t\in\RR_+$ and all $f\in\mathcal{C}_b(S)$
\begin{align*}
\begin{aligned}
\lim_{n\to\infty} P(t)f\left(1/n\right)
&=\lim_{n\to\infty}\left(e^{-nt}f\left(1/n\right)+\left(1-e^{-nt}\right)f\left(1\right)\right)\\
&=\mathbbm{1}_{\{0\}}(t)f(0)+\left( 1-\mathbbm{1}_{\{0\}}(t) \right)f(1)=P(t)f(0).
\end{aligned}
\end{align*}

Moreover, note that $\delta_1$ is the unique invariant probability measure of $\{P(t)\}_{t\in\mathbb{R}_+}$ such that $\mathtt{Int}( \mathrm{supp} (\delta_1))=\{1\}\neq\emptyset$. 
Since 
$\lim_{t\to\infty}e^{-nt}=0$ for all $n\in\NN$, and also 
$\lim_{t\to\infty}\mathbbm{1}_{\{0\}}(t)=0$, we obtain that 
$\{P(t)\}_{t\in\mathbb{R}_+}$ is asymptotically stable. Let us, however, show 
that it does not have the e-property in $\CC_b(S)$. Indeed, by definition, $\delta_0 P (t)= \delta_1$ for each $t> 0$. On the other hand, for any $n\in\NN$, we obtain
\begin{align*}
\delta_{1/n}P\left(1/n\right)=e^{-1}\delta_{1/n}+\left(1-e^{-1}\right)\delta_1,
\end{align*}
which implies that for any  $f \in \CC_b(S)$
\begin{align*}
\lim_{n\to\infty} \lvert P\left(1/n\right)f(1/n) - f(1) \rvert = e^{-1} \lvert f(0) - f(1) \rvert,
\end{align*}
and so $\{P(t)\}_{t\in\mathbb{R}_+}$ fails the e-property in $\CC_b(S)$ at 0. 
Due to Theorem \ref{thm:e-prop}, it however possesses the eventual e-property in $\CC_b(S)$.

As a~consequence, referring to Corollary \ref{thm_final}, one can deduce that  $\{P(t)\}_{t\in\mathbb{R}+}$ is not strongly stochastically continuous (in fact, one can even easily show that it is not stochastically continuous at zero).
\end{example}

\begin{example}\label{ex:3}
Let 
$S = \{0\}\cup\{n:\, n \in \NN\}\cup\{1/n:\,n\in\NN\}$, and note that, if endowed with the Euclidean metric, it is a Polish space. 
The idea is to define a Markov semigroup $\{P(t)\}_{t\in\mathbb{R}_+}$ so that it transforms measures in the following way: 
$$\delta_{\frac{1}{n}} \to \delta_n \to \delta_0 \to \delta_1 \to 
\delta_1\quad\text{for all}\quad n\in\NN, n \ge 2.$$ 
To achieve this, we introduce generators 
\begin{equation}
Q_n = \begin{pmatrix} 0 & 1 & 0 & 0 \\ 0 & -1 & n & 0 \\ 0 & 0 & -n & n \\ 0 & 0 
& 0 & -n
\end{pmatrix},\;\;\; n \ge 2,
\end{equation}
of embeddable stochastic matrices $\exp(Q_n t)$, $t\in\mathbb{R}_+$, that induce transition operators $P(t):S\times \BR(A)\to[0,1]$ in accordance with 
\begin{equation}\label{ex2_def}
\exp(Q_n t) = 
\begin{pmatrix} 
P(t)(1, 1) & P(t)(0, 1) & P(t)(n, 1) & P(t)(1/n, 
1) \\ 
P(t)(1, 0) & P(t)(0, 0) & P(t)(n, 0) & P(t)(1/n, 0) \\
P(t)(1, n) & P(t)(0, n) & P(t)(n, n) & P(t)(1/n, 
n) \\
P(t)(1, 1/n) & P(t)(0, 1/n) & P(t)(n, 1/n) & P(t)(1/n, 1/n)
\end{pmatrix},n \ge 2.
\end{equation}  
As a solution of \eqref{ex2_def}, and referring to \eqref{pi_P_U}, we see that for each $t\in\mathbb{R}_+$ a Markov operator $P(t):\mathcal{M}(S)\to\mathcal{M}(S)$ is determined by
\begin{align*}
&\delta_1P(t)=\delta_1,\qquad \delta_0P(t)=e^{-t}\delta_0+\left(1-e^{-t}\right)\delta_1,\\
&\delta_nP(t)=e^{-nt}\delta_n+\frac{n}{n-1}\left(e^{-t}-e^{-nt}\right) \delta_0
+a_{n,t}\delta_1,\\
&\delta_{1/n}P(t)=e^{-nt}\delta_{1/n}+nte^{-nt}\delta_n+\frac{n^2}{(n-1)^2}
\left(e^{-t}+e^{-nt}(t-nt-1)\right)\delta_0+b_{n,t}\delta_1
\end{align*}
for all $n\in\NN\backslash\{1\}$, where $a_{n,t},b_{n,t}\in\mathbb{R}$ are given by 
$(\delta_{n}P(t))(S)=1$ and $(\delta_{1/n}P(t))(S)=1$, respectively.

One can easily check that $\{P(t)\}_{t\in\mathbb{R}_+}$, defined as above, is indeed a Markov semigroup, and since $\text{w}\,\text{-}\lim_{n\to\infty}\delta_{1/n}P(t)=\delta_0P(t)$ for each $t\in\mathbb{R}_+$, it also possesses the Feller property. Moreover, it is not difficult to see that $\{P(t)\}_{t\in\mathbb{R}_+}$ is stochastically continuous (left and right) and asymptotically 
stable, with the unique invariant probability measure $\mu_*=\delta_1$ such that $\mathtt{Int} (
\mathrm{supp} (\mu_*) )\neq \emptyset$ (the assumptions of Theorem \ref{thm:e-prop} are thus satisfied). However, since for every $n\in\NN\backslash\{1\}$
$$\delta_{1/n}P(1/n) 
= e^{-1} \delta_{1/n} + e^{-1} \delta_n + 
\frac{n^2}{(n-1)^2}(e^{-1/n}+\frac1n e^{-1} -2e^{-1})\delta_0 + 
b_{n,1/n}\delta_1,$$ 
we obtain that for all $f \in \CC_b(S)$
$$\limsup_{n\to\infty} \lvert P(1/n)f(1/n) - P(1/n)f(0) \rvert = \limsup_{n\to\infty} e^{-1} \lvert f(n) - f(0) \rvert .$$
Hence, $\{P(t)\}_{t\in\mathbb{R}_+}$ does not have the e-property in $\CC_b(S)$, and thus, in view of Corollary \ref{thm_final}, it is not strongly stochastically continuous.
\end{example}

\subsection{Proofs}\label{sec:proofs}
Within this section we are going to present the proofs of the main results of this paper, that is, Theorems \ref{thm:e-prop},  \ref{thm:eventual_vs_e} and Lemma \ref{thm:e-prop_[0,T]}.  

Let us begin by formulating and proving two lemmas that will help us establish Theorem \ref{thm:e-prop}. The proofs of both these lemmas, as well as the proof of Theorem \ref{thm:e-prop} itself, draw on certain ideas employed in \cite{hille_szarek_ziem_17}.

\begin{lemma}\label{lem1}
Let $\{P(t)\}_{t\in\RR_+}$ be an asymptotically stable Markov-Feller semigroup, and let $\mu_*\in\mathcal{M}_1(S)$ be its unique invariant probability measure. 
If condition \eqref{notempty} holds, then  
for every $f\in\mathcal{C}_b(S)$ and any $\varepsilon>0$ there exist $T_0\in\mathbb{R}_+$ and an open ball $B\subset \mathrm{supp}\left(\mu_*\right)$ such that
\begin{align}\label{eq:lem1}
\left\lvert P(t)f(x)-\left\langle f,\mu_*\right\rangle\right\rvert\leq \varepsilon\quad \text{for all}\quad x\in B\quad \text{and all}\quad t\geq T_0.
\end{align}
\end{lemma}
\begin{proof}
Due to \eqref{notempty}, we know that there exists an open set $W\subset\mathrm{supp}\left(\mu_*\right)$ which is nonempty, and therefore we can define $Y:=\mathtt{Cl}(W)\subset S$. Note that as a~subspace of a Polish space, $Y$ is Polish too, and thus it is a~Baire space, which fact shall be used later in the proof.

Let $f\in\mathcal{C}_b(S)$, and let $\varepsilon>0$. For any $t\in\mathbb{R}_+$ define
\begin{align*}
Y_t:=\left\{x\in Y:\;\;\left\lvert P(s)f(x)-\left\langle 
f,\mu_*\right\rangle\right\rvert\leq \varepsilon\;\;\text{for all
}\;\; s\geq t\right\}.
\end{align*}
Since $\{P(t)\}_{t\in\mathbb{R}_+}$ is assumed to be Feller, that is, $P(s)f$ is continuous for any $s\in\mathbb{R}_+$, we see that the sets $Y_t$, $t\in\mathbb{R}_+$, are closed in $Y$ (as intersections of closed sets).

Let us now take a sequence $\{t_n\}_{n\in\mathbb{N}}$ of real numbers such that $0\leq t_1<t_2<\ldots$ and $\lim_{n\to\infty}t_n=\infty$. Referring to the fact that $\{P(t)\}_{t\in\mathbb{R}_+}$ is, by assumption, asymptotically stable, we obtain that 
$Y=\bigcup_{n\in\mathbb{N}}Y_{t_n}$. 
Note that, by definition, $\mathtt{Int}(Y)\neq\emptyset$, whence there exists $N\in\mathbb{N}$ such that 
\begin{align}
\label{eq:Int Y_t_N}
\mathtt{Int}(Y_{t_N})\neq\emptyset
\end{align}
(otherwise all the sets $Y_{t_n}$ would be nowhere dense, and so, by the Baire theorem, $Y=\bigcup_{n\in\mathbb{N}}Y_{t_n}$ would be a~boundary set, which would obviously contradict the nonemptiness of $W$). 

Finally, we observe that $\mathrm{supp}\left(\mu_*\right)$ is closed in $S$, whence $Y_{t_N}\subset \mathrm{supp}\left(\mu_*\right)$. Condition \eqref{eq:Int Y_t_N} then implies the existence of an open set $B\subset Y_{t_N}\subset \mathrm{supp}\left(\mu_*\right)$. Letting $T_0:=t_N$, we obtain the assertion of the lemma.
\end{proof}

\begin{lemma}\label{lem2}
Let $\mu\in\mathcal{M}_1(S)$ and $A\in\mathtt{Bor}(S)$. If $I$ is an infinite set 
of indexes, and $\{A_i\}_{i\in I}\subset\mathtt{Bor}(S)$ is a family 
of pairwise disjoint sets satisfying $A=\bigcup_{i\in I}A_i$, then there exists a 
countable subset $J$ of $I$ such that $\mu(A_i)=0$ for all $i\in I\backslash J$.
\end{lemma}

\begin{proof}
Define $J:=\{i\in I:\; \mu(A_i)>0\}$, and observe that for all $i\in J$ there 
exists $n\in\mathbb{N}$ for which $\mu(A_i)\in (1/(n+1),1/n]$. Since 
$\mu(A)<\infty$, for any 
$n\in\mathbb{N}$ only a~finite number of indexes $i\in J$ can satisfy 
$\mu(A_i)\in (1/(n+1),1/n]$.  As a~consequence, 
$J$ can be at most a~countable set. 
\end{proof}

The proof of Theorem \ref{thm:e-prop} then proceeds as follows:
\begin{proof}[Proof of Theorem \ref{thm:e-prop}] 
Fix $f\in\mathcal{C}_b(S)$ and $\varepsilon>0$  arbitrarily. Moreover, let $x_0\in S$, and let $\{x_i\}_{i\in\mathbb{N}}\subset S$ be an arbitrary sequence converging to  $x_0$.

By Lemma \ref{lem1}, there exist $T_0\in\mathbb{R}$ and an open ball $B\subset \textrm{supp}\left(\mu_*\right)$, say $B=B(z,r)$ for some $z\in \textrm{supp}\left(\mu_*\right)$ and some $r>0$, for which  \eqref{eq:lem1} holds with $\varepsilon/4$ (in the place of $\varepsilon$). Further, observe that $\gamma:=\mu_*(B(z,r))>0$, and choose $\alpha\in(0,\gamma)$. 
Using the fact that $\{P(t)\}_{t\in\mathbb{R}_+}$ is asymptotically stable, and referring to the Portmanteau theorem, we get
\begin{align}\label{eq:>alpha}
\liminf_{t\to\infty}\mu P(t)\left(B(z,r)\right)\geq \mu_*\left(B(z,r)\right)=\gamma>\alpha\quad \text{for any}\quad \mu\in\mathcal{M}_1(S).
\end{align}

Now, fix $k\in\mathbb{N}$ so that  
\begin{align}\label{eq:k}
4(1-\alpha)^k\|f\|_{\infty}<\varepsilon,
\end{align}
and let us define the sequences $\{t_l\}_{l=1}^k\subset\mathbb{R}_+$ 
and $\{\nu_l^{x_0}\}_{l=1}^k,\{\mu_l^{x_0}\}_{l=1}^k\subset\mathcal{M}_1(S)$ 
as follows:
\begin{enumerate}
\item Fix $t_1\in\mathbb{R}_+$ such that 
$$\delta_{x_0}P(t_1)(B(z,r))>\alpha,$$ 
the existence of which follows from  \eqref{eq:>alpha}, and note that 
\begin{align*}
\delta_{x_0}P\left(t_1\right)(B(z,r))=\lim_{n\to\infty}\delta_{x_0}P\left(t_1\right)\left(B\left(z,r-1/n\right)\right), 
\end{align*}
whence we can choose $n_*\in\mathbb{N}$ for which 
\begin{align*}
\delta_{x_0}P(t_1)(B(z,r-1/n_*))>\alpha.
\end{align*}
Further, Lemma \ref{lem2} implies that there is  $r_1\in[r-1/n_*,r)$ satisfying
\begin{align}\label{eq:=0}
\delta_{x_0}P\left(t_1\right)\left(\partial \left(B\left(z,r_1\right)\right)\right)=0.
\end{align}
In view of the above, we can define
\begin{align}\label{def:nu1}
\nu_1^{x_0}(\cdot):=
\frac{\delta_{x_0}P\left(t_1\right)\left(\cdot\cap B\left(z,r_1\right)\right)}{\delta_{x_0}P\left(t_1\right)\left(B\left(z,r_1\right)\right)}
\end{align}
and 
\begin{align}\label{def:mu1}
\mu_1^{x_0}:=\frac{1}{1-\alpha}\left(
\delta_{x_0}P\left(t_1\right)-\alpha\nu_1^{x_0}\right),
\end{align}
which are both {Borel probability measures on $S$}. 
Let us also note that, by the Feller property of $\{P(t)\}_{t\in\mathbb{R}_+}$, we have
\begin{align}\label{eq:Feller_conseq}
\text{w}\,\text{-}\lim_{i\to\infty}\delta_{x_i}P\left(t_1\right)=\delta_{x_0}P\left(t_1\right).
\end{align}
\item Let $l\in\{1,\ldots,k-1\}$. Suppose that the Borel probability measures $\nu_m^{x_0}$ and $\mu_m^{x_0}$ are all well-defined for $m\leq l$. Then there exists $t_{l+1}\in\mathbb{R}_+$ such that 
\hbox{$\mu_l^{x_0}P\left(t_{l+1}\right)(B(z,r))>\alpha$}. Using arguments similar as in Step 1, we can further conclude that there is $r_{l+1}<r$ such that
\begin{align*}
\mu_l^{x_0}P\left(t_{l+1}\right)\left(B\left(z,r_{l+1}\right)\right)>\alpha\quad\text{and}\quad
\mu_l^{x_0}P\left(t_{l+1}\right)\left(\partial \left(B\left(z,r_{l+1}\right)\right)\right)=0.
\end{align*}
This, in turn, allows us to define
\begin{align}\label{def:nul}
\nu_{l+1}^{x_0}(\cdot):=\frac{\mu_{l}^{x_0}P\left(t_{l+1}\right)\left(\cdot\cap B\left(z,r_{l+1}\right)\right)}{\mu_{l}^{x_0}P\left(t_{l+1}\right)\left(B\left(z,r_{l+1}\right)\right)}
\end{align}
and
\begin{align}\label{def:mul}
\mu_{l+1}^{x_0}:=\frac{1}{1-\alpha}\left(\mu_l^{x_0}P\left(t_{l+1}\right)-\alpha\nu_{l+1}^{x_0}\right),
\end{align}
both belonging to $\mathcal{M}_1(S)$. 
\item 
Proceeding as in Step 2, we end up with the  inductively defined sequences $\{t_l\}_{l=1}^k\subset\mathbb{R}_+$ 
and $\{\nu_l^{x_0}\}_{l=1}^k,\{\mu_l^{x_0}\}_{l=1}^k\subset\mathcal{M}_1(S)$. 
Let us indicate that for any $t\geq t_1+\ldots+t_k$ we get
\begin{align}\label{sum_for_x0}
\begin{aligned}
\delta_{x_0}P(t)
=&\delta_{x_0}P\left(t_1\right)P\left(t-t_1\right)\\
=&\alpha \nu_1^{x_0}P\left(t-t_1\right)+(1-\alpha)\mu_1^{x_0}P\left(t-t_1\right)=\ldots\\
=&\alpha \nu_1^{x_0}P\left(t-t_1\right)+\ldots\\
&+\alpha(1-\alpha)^{k-1}\nu_k^{x_0}P\left(t-t_1-\ldots-t_k\right)\\
&+(1-\alpha)^k\mu_k^{x_0}P\left(t-t_1-\ldots-t_k\right).
\end{aligned}
\end{align}
\end{enumerate}

Combining observation \eqref{eq:Feller_conseq} and the Portmanteau theorem, we see that there exists $i_1\in\mathbb{N}$ such that
\begin{align*}
\delta_{x_i}P\left(t_1\right)\left(B\left(z,r_1\right)\right)>\alpha\quad\text{for all}\quad i\geq i_1,
\end{align*}
and thus, for any $i\geq i_1$, we can define $\nu_1^{x_i},\mu_1^{x_i}\in\mathcal{M}_1(S)$, exactly as in \eqref{def:nu1} and \eqref{def:mu1}, respectively, but with $x_i$ in the place of $x_0$.  

Using \eqref{eq:Feller_conseq} and the Portmanteau theorem again, we observe that, for every open subset $O$ of $S$,
\begin{align*}
\liminf_{i\to\infty}\delta_{x_i}P\left(t_1\right)\left(O\cap B\left(z,r_1\right)\right)
\geq \delta_{x_0}P\left(t_1\right)\left(O\cap B\left(z,r_1\right)\right).
\end{align*}
This, in turn, implies that
\begin{align*}
\text{w}\,\text{-}\lim_{i\to\infty}\delta_{x_i}P\left(t_1\right)\left(\cdot\cap B\left(z,r_1\right)\right)=\delta_{x_0}P\left(t_1\right)\left(\cdot\cap B\left(z,r_1\right)\right),
\end{align*}
and, as a consequence, $\text{w}\,\text{-}\lim_{i\to\infty}\nu_1^{x_i}=\nu_1^{x_0}$ and $\text{w}\,\text{-}\lim_{i\to\infty}\mu_1^{x_i}=\mu_1^{x_0}$.

For any sufficiently large $i$, we can define $\{\nu_l^{x_i}\}_{l=1}^k, \{\mu_l^{x_i}\}_{l=1}^k\subset\mathcal{M}_1(S)$ inductively on $l$ as follows. Let $l\in\{1,\ldots,k-1\}$, and suppose that there is $i_l\in\mathbb{N}$ such that, for any $i\geq i_l$, the measures $\nu_1^{x_i},\ldots,\nu_l^{x_i}$ and $\mu_1^{x_i},\ldots,\mu_l^{x_i}$ are all well-defined Borel probability measures satysfying $\text{w}\,\text{-}\lim_{i\to\infty}\nu_l^{x_i}=\nu_l^{x_0}$ and \linebreak$\text{w}\,\text{-}\lim_{i\to\infty}\mu_l^{x_i}=\mu_l^{x_0}$. The Feller property of $\{P(t)\}_{t\in\mathbb{R}_+}$ then gives the following: \linebreak$\text{w}\,\text{-}\lim_{i\to\infty}\mu_l^{x_i}P(t_{l+1})=\mu_l^{x_0}P(t_{l+1})$. 
Moreover, due to the choice of $r_{l+1}$, for which 
$\mu_l^{x_0}P\left(t_{l+1}\right)\left(\partial \left(B\left(z,r_{l+1}\right)\right)\right)=0$, 
and, again, the Portmanteau theorem, we get
\begin{align*}
\lim_{i\to\infty}\mu_l^{x_i}P\left(t_{l+1}\right)\left(B\left(z,r_{l+1}\right)\right)
=\mu_l^{x_0}P\left(t_{l+1}\right)\left(B\left(z,r_{l+1}\right)\right).
\end{align*}
Repeating the arguments already presented above, we can claim that there exists $i_{l+1}\geq i_l$ such that for all $i\geq i_{l+1}$
\begin{align*}
\mu_l^{x_i}P\left(t_{l+1}\right)\left(B\left(z,r_{l+1}\right)\right)>\alpha\quad\text{and}\quad \nu_{l+1}^{x_i}\in\mathcal{M}_1(S).
\end{align*}
Moreover,
\begin{align*}
\text{w}\,\text{-}\lim_{i\to\infty}\mu_l^{x_i}P\left(t_{l+1}\right)\left(\cdot\cap B\left(z,r_{l+1}\right)\right)=\mu_l^{x_0}P\left(t_{l+1}\right)\left(\cdot\cap B\left(x,r_{l+1}\right)\right),
\end{align*}
whence $\mu_{l+1}^{x_i}\in\mathcal{M}_1(S)$ for every $i\geq i_{l+1}$ and also 
$\text{w}\,\text{-}\lim_{i\to\infty}\nu_{l+1}^{x_i}=\nu_{l+1}^{x_0}$, \linebreak$\text{w}\,\text{-}\lim_{i\to\infty}\mu_{l+1}^{x_i}=\mu_{l+1}^{x_0}$.

In the final part of the proof, let us observe that for every $i\geq i_k$ and every $t\geq t_1+\ldots+t_k$ we have
\begin{align}\label{eq:sum}
\begin{aligned}
\delta_{x_i}P(t)
=&\alpha\nu_1^{x_i}P\left(t-t_1\right)
+\ldots+\alpha(1-\alpha)^{k-1}\nu_k^{x_i}P\left(t-t_1-\ldots -t_{k}\right)\\
&+(1-\alpha)^k\mu_k^{x_i}P\left(t-t_1-\ldots-t_k\right).
\end{aligned}
\end{align}
Since, for every $l\in\{1,\ldots,k\}$ and every $i\geq i_k$, $\textrm{supp}\left(\nu_l^{x_i}\right),\textrm{supp}\left(\nu_l^{x_0}\right)\subset\overline{B\left(z,r_l\right)}\subset B(z,r)$, we obtain, by Lemma \ref{lem1}, that for each $l\in\{1,\ldots,k\}$, each $i\geq i_k$ and each $t\geq T_0$
\begin{align}\label{estim}
\begin{aligned}
\left\lvert
\left\langle f,\nu_l^{x_i}P(t)-\nu_l^{x_0}P(t)\right\rangle\right\rvert
=&\left\lvert
\left\langle P(t)f,\nu_l^{x_i}\right\rangle
-\left\langle P(t)f,\nu_l^{x_0}\right\rangle\right\rvert\\
\leq &\left\lvert
\left\langle P(t)f-\left\langle f,\mu_*\right\rangle,\nu_l^{x_i}\right\rangle\right\rvert
+\left\lvert
\left\langle P(t)f-\left\langle f,\mu_*\right\rangle,\nu_l^{x_0}\right\rangle\right\rvert\leq \frac{\varepsilon}{2}.
\end{aligned}
\end{align}

Now, let $\tau:=t_1+\ldots+ t_k+T_0$ and $I:=i_k$. Referring subsequently to \eqref{sum_for_x0}, \eqref{eq:sum}, \eqref{estim} and \eqref{eq:k}, we get
\begin{align*}
\sup_{t\geq \tau} \left\lvert P(t)f\left(x_i\right)-P(t)f\left(x_0\right)\right\rvert
&=
\sup_{t\geq \tau} \left\lvert\left\langle f,\delta_{x_i} P(t)-\delta_{x_0} P(t)\right\rangle\right\rvert\\
&\leq
\sup_{t\geq T_0}
\frac{\varepsilon}{2}\left(\alpha+\alpha(1-\alpha)+\ldots+\alpha(1-\alpha)^{k-1}\right)\\
&\quad+(1-\alpha)^k
\left\lvert \left\langle f,\mu_k^{x_i}P(t)-\mu_k^{x_0}P(t)\right\rangle\right\rvert\\
&\leq \frac{\varepsilon}{2}+2(1-\alpha)^k\|f\|_{\infty}< \varepsilon\quad\text{for all}\quad i\geq I,
\end{align*}
and thus the proof is completed.
\end{proof}

Now, let us proceed with the proof of Lemma \ref{thm:e-prop_[0,T]}. 
\begin{proof}[Proof of Lemma \ref{thm:e-prop_[0,T]}]

Let $z\in S$. Suppose (contrary to our claim) that there exist $T\in\mathbb{R}_+$, $f\in\mathcal{C}_b(S)$ and  sequences $\{t_k\}_{k\in\NN}$, $\{x_k\}_{k\in \NN}$, of points in $[0,T]$ and $S$, respectively, where the latter converges to $z$, such that
\begin{equation}\label{eq:contrary}
\liminf_{k\to\infty}
\left\lvert 
P\left( t_k \right) f \left( x_k \right)
-P\left( t_k \right) f(z)
\right\rvert > \varepsilon\;\;\;\text{for some}\;\;\;\varepsilon > 0.
\end{equation}
Since the indicated sequence $\{t_k\}_{k\in\NN}$ is bounded, we can choose its subsequence (for the sake of simplicity of notation also denoted by $\{t_k\}_{k\in\NN}$) which converges monotonically to a certain point in $[0,T]$, say $t_0$. Let us now consider the following two cases: (a) $t_k\to t_0^+$ and (b) $t_k\to t_0^-$. 
\begin{itemize}
\item[(a)] According to \eqref{eq:contrary}, there exists $K\in\mathbb{N}$ such that for every $k\geq K$ we have
\begin{align}\label{eq:first}
\begin{aligned}
\varepsilon 
\le & 
\left\lvert 
P\left( t_k \right) f \left( x_k \right)
-P\left( t_k \right) f(z)
\right\rvert\\
\leq &
\left\lvert 
P\left( t_k-t_0 \right) P\left( t_0 \right) f \left( x_k \right) - P\left( t_0 \right) f \left( x_k \right)
\right\rvert
+
\left\lvert
P\left( t_0 \right) f \left( x_k \right) - P\left( t_0 \right) f \left( z \right)
\right\rvert \\
&+
\left\lvert 
P\left( t_k-t_0 \right) P\left( t_0 \right) f \left( z \right) - 
P\left( t_0 \right) f \left( z \right)
\right\rvert,
\end{aligned}
\end{align}
where the semigroup property of $\{P(t)\}_{t\in\mathbb{R}_+}$ has also been used above. 
Note that, due to the Feller property, $P(t_0)f\in\mathcal{C}_b(S)$, and therefore the second term on the right-hand side of \eqref{eq:first} converges to zero as $k\to\infty$. Moreover, observe that $u_k:=t_k-t_0\to 0^+$, and thus, referring to strong stochastic continuity of $\{P(t)\}_{t\in\mathbb{R}_+}$, we see that the first and the last therm on the right-hand side of \eqref{eq:first} also vanish if $k\to\infty$. The above, however, contradicts \eqref{eq:contrary}.
\item[(b)] According to \eqref{eq:contrary}, there exists $K\in\mathbb{N}$ such that for every $k\geq K$ we have
\begin{align}\label{eq:second}
\begin{aligned}
\varepsilon 
\leq &
\left\lvert 
P\left( t_k \right) f \left( x_k \right) - P\left( t_k \right)P\left( t_0 - t_k \right) f \left( x_k \right)
\right\rvert
+
\left\lvert
P\left( t_0 \right) f \left( x_k \right) - P\left( t_0 \right) f \left( z \right)
\right\rvert \\
&+
\left\lvert 
P\left( t_k\right) f \left( z \right) - 
P\left( t_k \right)P\left( t_0 - t_k \right) f \left( z \right)
\right\rvert,
\end{aligned}
\end{align}
where the semigroup property of $\{P(t)\}_{t\in\mathbb{R}_+}$ has also been used above. Due to the Feller property,  the second term on the right-hand side of \eqref{eq:second} converges to zero as $k\to\infty$. 
Let us further observe that 
\begin{align*}
\left\lvert 
P\left( t_k\right) f \left( z \right) - 
P\left( t_k \right)P\left( t_0 - t_k \right) f \left( z \right)
\right\rvert
&=
\left\lvert 
P\left( t_k\right) \left( f - 
P\left( t_0 - t_k \right) f\right) \left( z \right)
\right\rvert\\
&\leq \left\|f - 
P\left( t_0 - t_k \right) f\right\|_{\infty},
\end{align*} 
which expression converges to zero as $k\to\infty$, thanks to the assumed strong stochastic continuity. Using the same arguments, we see that the first term  on the right-hand side of \eqref{eq:second} also converges to zero as $k\to\infty$. The above, however, contradicts \eqref{eq:contrary}.
\end{itemize}
The proof is completed.
\end{proof}

Finally, employing Lemma \ref{thm:e-prop_[0,T]}, we can establish Theorem \ref{thm:eventual_vs_e}.

\begin{proof}[Proof of Theorem \ref{thm:eventual_vs_e}]
Let $\varepsilon>0$. Moreover, choose $z\in S$ and a sequence $\{x_k\}_{k\in\NN}$ converging to it. 
Obviously, for any $k\in\mathbb{N}$ and any $T\in\mathbb{R}_+$
\begin{align}
\label{eq:2sup}
\begin{aligned}
\sup_{t\in\mathbb{R}_+}
\left\lvert P(t)f\left( x_k \right)-P(t)f(z)\right\rvert
\leq &\sup_{t\in[0,T]}
\left\lvert P(t)f\left( x_k \right)-P(t)f(z)\right\rvert\\
&+
\sup_{t\geq T}\left\lvert P(t)f\left( x_k \right)-P(t)f(z)\right\rvert.
\end{aligned}
\end{align} 
We know, by the assumed eventual e-property in $\CC_b(S)$, that for every $f\in\mathcal{C}_b(S)$ 
there exist $\tau\in\mathbb{R}_+$ and $K\in\NN$ such that for every $k\geq K$  
we have
\begin{align}
\label{eq:eventual}
\sup_{t\geq \tau}\left\lvert P(t)f\left( x_k \right)-P(t)f(z)\right\rvert <\varepsilon/2.
\end{align}
Moreover, referring to Lemma \ref{thm:e-prop_[0,T]}, we obtain that there exists $M\in\NN$ such that for all $k\geq M$ 
\begin{equation}\label{eq:thm[0,T]}
\sup_{t\in[0,\tau]}
\left\lvert P(t)f\left( x_k \right)-P(t)f(z)\right\rvert <\varepsilon/2.
\end{equation}
Now, combining \eqref{eq:eventual} and \eqref{eq:thm[0,T]}, we see that for all $k\geq \max\{K,M\}$ the right-hand side of \eqref{eq:2sup} (with $T=\tau$) is smaller than $\varepsilon$, which completes the proof.
\end{proof}

\section*{Appendix I: Stochastic continuity of Markov semigroups with compact state spaces}\label{appendix}

The purpose of this appendix is to demonstrate that the notions of stochastic 
continuity and strong stochastic continuity of Markov semigroups coincide when the 
underlying state space is a compact metric space (Theorem 
\ref{thm_append}). 
This statement, however, is not true if we weaken 
the assumption on the given space, e.g. if this space is only locally compact (cf. Example \ref{ex:3}).

In what follows for all $\mu_1,\mu_2 \in \MM_1(S)$, $f \in \CC_b(S)$ and $\epsilon>0$ we will write 
$\mu_1 \sim_{f,\epsilon} \mu_2$ 
whenever $\lvert\langle f, 
\mu_1 - \mu_2 \rangle\rvert < \epsilon$.

Before proceeding to the desired theorem, let us first establish the following lemma (which shall be useful in the proof of Theorem \ref{thm_append}):

\begin{lemma}\label{lem-first}
Let $S$ be a Polish space, and let $\mu,\mu_1,\mu_2,\ldots \in \MM_1(S)$. Moreover, assume that $\{P(t)\}_{t\in\mathbb{R}_+}$ is a 
stochastically continuous at zero Markov-Feller semigroup. If $\text{w}\,\text{-}\lim_{n\to\infty}\mu_n = \mu$, then
\begin{equation*}
\forall_{\substack{\theta > 0\\ \epsilon > 0\\ f \in \CC_b(S)}} \;\;
\exists_{\substack{t_* \in (0, \theta) \\ \delta > 0 \\ n_0 \in \NN}}\;\;
\forall_{\substack{t\in B(t_*,\delta)\\ n \ge 
n_0 }}
\quad \mu_n P(t) \sim_{f,\epsilon} \mu.
\end{equation*}
\end{lemma}
\begin{proof}
Fix $\theta > 0$, $\epsilon > 0$ and $f \in \CC_b(S)$. Further, choose  $\theta_0 \in (0, \theta)$ such that $\mu P(t) \sim_{f,\epsilon/2} \mu$ for all $t \le \theta_0$, which can be done thanks to the assumed stochastic continuity at zero of $\{P(t)\}_{t\in\mathbb{R}_+}$. 

For any $t \in [0,\theta]$ let $n_t \in \NN$ denote the smallest number such that for any  $n \ge n_t$ we have 
$\mu_n P(t) \sim_{f,\epsilon/2} \mu$. Next, for every $n\in\NN$ define 
$A_n := \{t \in [0,\theta_0]: n_t\leq n\}$. Observe that $A_n \subset A_{n+1}$ and $\bigcup_{n\in\NN} 
A_n = [0, \theta_0]$. Obviously, $[0, \theta_0]$ is a second category set, whence, there exists $n_0 \in \NN$ for which $\mathtt{Int} (\mathtt{Cl} ( A_{n_0})) \neq \emptyset$. In this connection, we can assume that there are some $t_*\in(0,\theta_0)$ and $\delta>0$ such that $B(t_*, 
\delta) \subset \mathtt{Cl} ( A_{n_0}) $. 

Fix $t \in B(t_*, \delta)$ and $n \ge n_0$. Define $B_n := B(t_*, \delta) \cap A_n 
$. The set $B_n$ is dense in $B(t_*, \delta)$. We can choose a sequence $(t_k)_{k \in \NN} 
\subset B_n$ so that $t_k \to t^+$ and an index  $k_0 \in \NN$ for which 
$\mu_n P(t_{k_0})\sim_{f,\epsilon/2}  \mu_n P(t)$. Then, we finally get
\begin{equation*}
\left\lvert\left\langle f, \mu_n P(t) - \mu \right\rangle\right\rvert 
\le \left\lvert\left\langle f, \mu_n P(t) -  
\mu_n P\left({t_{k_0}}\right) \right\rangle\right\rvert  + 
\left\lvert\left\langle f,  \mu_n P\left(t_{k_0}\right) - \mu \right\rangle\right\rvert  <  \epsilon,
\end{equation*}
which completes the proof.
\end{proof}

\begin{theorem}\label{thm_append}
If $S$ is a compact metric space and $\{P(t)\}_{t\in\mathbb{R}_+}$ is a stochastically continuous at zero Markov-Feller semigroup, then $\{P(t)\}_{t\in\mathbb{R}_+}$ is also strongly stochastically continuous.
\end{theorem}
\begin{proof}
Suppose (contrary to our claim) that $\{P(t)\}_{t\in\mathbb{R}_+}$ is not strongly stochastically continuous. Then, there are $f \in \CC_b(S)$, $\varepsilon > 0$, $N\in\NN$, $\{x_n\}_{n\in\NN} \subset S$ and \hbox{$\{t_n\}_{n\in\NN}\subset\mathbb{R}_+$} converging to $0$  such that for all $n\geq N$
\begin{equation}\label{eq:proof_append}
\left\lvert P\left(t_n\right) f\left(x_n\right)  - f\left(x_n\right)\right\rvert > 
\varepsilon.
\end{equation}
By compactness of $S$, 
we can choose a subsequence of $\{x_n\}_{n\in\NN}$ (for the sake of simplicity of notation also denoted by $\{x_n\}_{n\in\NN}$) which is convergent to some $x\in S$ and for which $\text{w}\,\text{-}\lim_{n\to\infty} \delta_{x_n} P(t_n) = \mu$ for some $\mu\in\mathcal{M}_1(S)$, where the latter is a~consequence of the Prokhorov theorem. Combining it with \eqref{eq:proof_append}, we get
\[ 
\left\lvert 
\left\langle f,\delta_{x_n} P(t_n)-\mu\right\rangle 
\right\rvert
+\left\lvert 
\left\langle f,\mu\right\rangle-f(x)\right\rvert
+\left\lvert f(x)-
f\left(x_n\right)\right\rvert 
\geq \left\lvert P\left(t_n\right) f\left(x_n\right)  - f\left(x_n\right)\right\rvert
>\varepsilon,  
\]
hence, after moving to the limit (with $n\to\infty$) on the right and left sides of the above inequality, we obtain
\begin{equation}\label{eq-first}
\lvert\langle f, \mu \rangle - f(x)\rvert \ge \varepsilon.
\end{equation}

Then, take $\theta > 0$ such that for all $t \le 
\theta$ we have $ \delta_x P(t) \sim_{f,\varepsilon/3} \delta_x$, which can be done thanks to the assumed stochastic continuity at zero of $\{P(t)\}_{t\in\mathbb{R}_+
}$. For all $n\in\NN$ let us introduce $\mu_n:=\delta_{x_n} P(t_n)$. By Lemma~\ref{lem-first}, there are $t_* 
\in (0, \theta)$, $\delta > 0$ and $n_0 \in \NN$  such that $\mu_n P(t) \sim_{f,\varepsilon/3} 
\mu$ for all $t\in B(t_*,\delta)$ and all $n \ge n_0$. It now suffices to observe that there exists  
$n_1 \ge \max\{n_0,N\}$ for which $t_{n_1}\leq \min\{t_*,\delta\}$ and $\delta_{x_{n_1}} P(t_*)
\sim_{f,\epsilon/3} \delta_x P(t_*)$. Hence,  we finally get the following:
\begin{align*}
&\delta_x \sim_{f,\epsilon/3} \delta_x P\left(t_*\right),\qquad 
\delta_x P\left(t_*\right) \sim_{f,\epsilon/3} 
\delta_{x_{n_1}} P\left(t_*\right) 
,\\
& \delta_{x_{n_1}} P\left(t_*\right)
=\delta_{x_{n_1}} P\left(t_{n_1}\right) P\left(t_*-t_{n_1}\right)
= \mu_{n_1} P\left(t_*-t_{n_1}\right)
\sim_{f,\epsilon/3} 
\mu,
\end{align*}
This, however, gives us  $\delta_x \sim_{f,\epsilon} \mu$, and thus contradicts \eqref{eq-first}. The proof is completed.
\end{proof}

\section*{Appendix II: Relations  between the notions of the e-property in $R$ for different sets $R$}\label{appendix2}

By $\LL_{bs}(S)$ let us denote the subfamily of $\LL_b(S)$ that consists of non-negative Lipschitz continuous functions with bounded supports. 

In this appendix, we present an analysis of various definitions of the \hbox{e-property} in $R$, where $R \in \{ \LL_{bs}(S), \LL_{b}(S), \CC_b(S) \}$, which are non-equivalent in general, as demonstrated in the examples below.


\begin{example}
    Let $S = [1, \infty)$ be a given Polish space equipped with the Euclidean metric, and let $\{P(t)\}_{t\in\mathbb{R}_+}$ be a Markov semigroup determined by $\delta_x P(t) = \delta_{x+t}$ for $t \in\mathbb{R}_+$ and $x \ge 1$. Using the similar reasoning as presented in \cite[Remark 2.1]{kukulski_wojewodka} we can show that $\{P(t)\}_{t\in\mathbb{R}_+}$ has the e-property in $\LL_b(S)$, but does not have the e-property in $\CC_b(S)$ at any point $z \in S$.
    \end{example}

    \begin{example}
    Let $S = [1, \infty)$ be a given Polish space equipped with the Euclidean metric, and let $\{P(t)\}_{t\in\mathbb{R}_+}$ be a Markov semigroup determined by $\delta_x P(t) = \delta_{xe^t}$ for $t \ge 0$ and $x \ge 1$. Take arbitrary $f \in \LL_{bs}(S)$ and $z \in S$. Let $L_f$ be a Lipschitz constant of $f$. There is $x_0 \ge z$, such that $f(x) = 0$ for all $x \ge x_0$. Therefore, we get 
    \begin{equation*}
    \begin{split}
\sup_{t \in\mathbb{R}_+} \,\lvert P(t)f(x)-P(t)f(z)\rvert
&=\sup_{t \in\mathbb{R}_+} \,\left\lvert f\left(xe^t\right)-f\left(ze^t\right)\right\rvert\\
&=\sup_{\ln(x_0/(z-0.5)) \geq t \geq 0} \,\lvert f(xe^t)-f(ze^t)\rvert\\
&\le\sup_{\ln(x_0/(z-0.5)) \geq t \geq 0} \, L_f\, e^t \lvert x-z\rvert = \frac{L_f x_0}{z-0.5}\, \lvert x-z\rvert
\end{split}
\end{equation*}
    for any $x\in[z-0.5,z+0.5]$, and thus
\begin{equation*}
\lim_{x \to z}\sup_{t \in\mathbb{R}_+} \,\lvert P(t)f(x)-P(t)f(z) \rvert= 0.
\end{equation*} 
We therefore see that $\{P(t)\}_{t\in\mathbb{R}_+}$ has the e-property in $\LL_{bs}(S)$. 

To prove that $\{P(t)\}_{t\in\mathbb{R}_+}$ does not have the e-property in $\LL_b(S)$ at any $z \in S$, let us define \hbox{$f \in \LL_b(S)$} as follows:
\begin{equation*}
    \begin{split}
        f(x) &= 2(x-1), \quad x \in [1,1.5),\\
        f(x) &= -2(x-2), \quad x \in [1.5,2),\\
        f(x) &= f(x-1), \quad x \ge 2.\\
    \end{split}
\end{equation*}
Further, let $z\in S$ be chosen arbitrarily. Consider a sequence $\{x_n\}_{n\in\NN}$  of points in $S$ such that $x_n = (1+1/(2n))z$ for every $n \in \NN$, and note  that it converges to $z$. Moreover, consider a sequence $\{t_n\}_{n\ge z}\subset \mathbb{R}_+$, where  $t_n = \ln(n/z)$ for $n\ge z$. We then get 
\begin{equation*}
    \lvert P(t_n)f(x_n)-P(t_n)f(z)\rvert = \left\lvert f\left(x_ne^{t_n}\right) - f\left(ze^{t_n}\right) \right\rvert = \lvert f(n+0.5) - f(n) \rvert = 1, 
\end{equation*}
and thus $\{P(t)\}_{t\in\mathbb{R}_+}$ fails the e-property in $\LL_b(S)$ at any $z\in S$.
\end{example}

 Before we present conditions under which certain definitions of the e-property coincide, let us provide an extension of the Portmanteau theorem.
\begin{lemma}\label{lem:weak-b-bs}
    Let $S$ be a Polish metric  space, and let $\mu\in \MM_1(S)$, $\{\mu_t\}_{t\in\mathbb{R}_+}\subset\MM_1(S)$. Then  $\text{w}\,\text{-}\lim_{t\to\infty}\mu_t = \mu$ if and only if $\lim_{t \to 	\infty}\bk{f}{\mu_t}= \bk{f}{\mu}$ for any $f\in \LL_{bs}(S)$.
\end{lemma}
\begin{proof}
    The direct implication obviously follows from the definition of weak convergence, hence we shall only prove the reverse implication. To do so, we will show that $\limsup_{t\to\infty}\mu_t(F)\leq\mu(F)$ for all closed sets $F\subset S$.
For any $\epsilon > 0$ and an arbitrary closed set $F \subset S$ let us define a closed set $$F_\epsilon \coloneqq \{ x \in S: \text{dist}(x,F) \le \epsilon\}.$$ Moreover, let us define the functions $f_{F,\epsilon} \in \LL_b(S)$ as $f_{F, \epsilon}(x) = 1 - \text{dist}(x,F) / \epsilon$ for $x \in F_\epsilon$ and $f_{F, \epsilon}(x) = 0$ for $x \not\in F_\epsilon$. Observe that $f_{F,\epsilon} \in \LL_{bs}(S)$ if $F$ is a bounded subset of $S$.

Let $\epsilon > 0$. If $F \subset S$ is closed and bounded, then 
\begin{equation}\label{sth}
    \limsup_{t\to\infty}\mu_t(F) \le \limsup_{t\to\infty} \bk{f_{F,\epsilon}}{\mu_t} = \bk{f_{F,\epsilon}}{\mu} \le \mu(F_\epsilon),
\end{equation}
and, in view of the fact that $\lim_{\epsilon \to 0} \mu(F_\epsilon) = \mu(F) $, we obtain the desired inequality. 
Now, let us take an arbitrary closed set $F \subset S$. Define a closed ball $B$ such that $\mu(B) \ge 1 - \epsilon$. Then,
\begin{equation*}
    1-\epsilon \le \bk{f_{B,\epsilon}}{\mu} = \lim_{t \to \infty} \bk{f_{B,\epsilon}}{\mu_t} \le \liminf_{t \to \infty} \mu_t(B_\epsilon),
\end{equation*}
which, due to \eqref{sth}, leads to
\begin{equation*}
    \limsup_{t\to\infty}\mu_t(F) \le \limsup_{t\to\infty}\mu_t(F \cap B_\epsilon) + \limsup_{t\to\infty}\mu_t(F \cap B'_\epsilon) \le \mu(F \cap B_\epsilon) + \epsilon.
\end{equation*}
Eventually, we get $\limsup_{t\to\infty}\mu_t(F) \le \mu(F) + \epsilon$ for any $\epsilon > 0$.
\end{proof}

A regular Markov semigroup $\{P(t)\}_{t\in\mathbb{R}_+}$ which is stochastically continuous at zero is \emph{right stochastically continuous}, that is, $\lim_{t\to t_0^+}P(t)f(x)=P(t_0)f(x)$ 
for all $x\in S$, all $t_0 \in\mathbb{R}_+$ and all $f\in \CC_b(S)$. If additionally $\{P(t)\}_{t\in\mathbb{R}_+}$ has the \hbox{e-property} in $\LL_b(S)$, we conclude that it is stochastically continuous \cite[Proposition 3.6.6]{ziemlanska_21}. In the following lemma we show that the same conclusion can be achieved if $\{P(t)\}_{t\in\mathbb{R}_+}$ has the e-property in $\LL_{bs}(S)$. 

\begin{lemma}\label{lem:stoch-bs}
    Let $\{P(t)\}_{t\in\mathbb{R}_+}$ be a regular Markov semigroup which is stochastically continuous at zero  and which has the e-property in $\LL_{bs}(S)$. Then $\{P(t)\}_{t\in\mathbb{R}_+}$ is stochastically continuous.
\end{lemma}
\begin{proof}
    Let $z \in S$, $\epsilon > 0$ and $f \in \LL_{bs}(S)$. Since $\{P(t)\}_{t\in\mathbb{R}_+}$ is assumed to have the e-property in $\LL_{bs}(S)$, there is $B(z,r)$ with some $r >0$ such that \hbox{$\lvert P(t)f(x) - P(t)f(z) \rvert \le \epsilon$} for any $t \in\mathbb{R}_+$ and $x \in B(z,r)$. Additionally, by the stochastic continuity at zero, we have $\text{w}\,\text{-}\lim_{ t \to 0^+} \delta_z P(t) = \delta_z$, an thus, by the Portmanteau theorem, there exists $\tau > 0$ such that $\delta_z P(t)(B(z, r)) > 1 - \epsilon$ for every $t \le \tau$. Therefore, for any $s, t \in \mathbb{R}_+$ satisfying $\tau \ge s - t \ge 0$, we obtain
    \begin{equation}
    \begin{split}
        \lvert P(s) f (z) - P(t)f(z) \rvert  =& \lvert \bk{P(t) f }{\delta_zP(s-t)} - P(t)f(z) \rvert \\ \le & \left\lvert \int_{B(z,r)} P(t) f(x) \delta_zP(s-t)(dx) - P(t)f(z) \right\rvert + \|f\|_\infty \epsilon \\
        \le &  \int_{B(z,r)}\left\lvert P(t) f(x) - P(t) f(z)\right\rvert \delta_zP(s-t)(dx)  \\+& \left\lvert P(t) f(z) \delta_zP(s-t)(B(z,r)) - P(t)f(z) \right\rvert + \|f\|_\infty \epsilon \\ \le & \epsilon + 2 \|f\|_\infty \epsilon.
    \end{split}
    \end{equation}
    Eventually, by Lemma~\ref{lem:weak-b-bs}, the function $t \mapsto P(t)f(z)$ is continuous for any $f \in \CC_b(S)$.
\end{proof}

\begin{theorem}
    A regular stochastically continuous at zero Markov semigroup $\{P(t)\}_{t\in\mathbb{R}_+}$ that is asymptotically stable and has the e-property in $\LL_{bs}(S)$ possesses also the e-property in $\CC_b(S)$.
\end{theorem}
\begin{proof}
    Assume, by the contradiction, that there is $z \in S$, $f \in \CC_b(S)$, $\epsilon>0$, \hbox{$\{t_n\}_{n \in \NN} \subset \RR_+$} and $\{x_n\}_{n \in \NN}\subset S$ converging to $z$ such that
    \begin{equation}\label{eq:thm-e-prop-eq}
        \lvert P(t_n) f(x_n) - P(t_n) f(z) \rvert > \epsilon.
    \end{equation}
    If the sequence $\{t_n\}_{n \in \NN}$ is bounded, let us take its subsequence $\{t_{k_n}\}_{n\in\NN}$ converging to some $t_0\in\mathbb{R}_+$. As $\{P(t)\}_{t\in\mathbb{R}_+}$ is stochastically continuous at zero and has the \hbox{e-property} in $\LL_{bs}(S)$, by Lemma~\ref{lem:stoch-bs}, it is also stochastically continuous. Hence, we get $\text{w}\,\text{-}\lim_{n\to\infty}\delta_{z}P(t_{k_n}) = \delta_{z}P(t_0)$. Otherwise, let us take a subsequence $\{t_{k_n}\}_{n\in\NN}$ of $\{t_n\}_{n \in \NN}$ that is divergent to $\infty$. As $\{P(t)\}_{t\in\mathbb{R}_+}$ is asymptotically stable (with the unique invariant probability measure $\mu_*$) we get \linebreak$\text{w}\,\text{-}\lim_{n\to\infty}\delta_{z}P(t_{k_n}) = \mu_*$.

    Let us take an arbitrary $g \in \LL_{bs}(S)$, and let us denote by $\mu$ either $\delta_z P(t_0)$ or $\mu_*$ (depending on the case that we consider). We have
    \begin{equation}\label{last}
        \left\lvert P\left(t_{k_n}\right)g\left(x_{k_n}\right) - \bk{g}{\mu} \right\rvert 
        \le \left\lvert P\left(t_{k_n}\right)g\left(x_{k_n}\right) - P\left(t_{k_n}\right)g(z) \right\rvert + \left\lvert P\left(t_{k_n}\right)g(z) - \bk{g}{\mu} \right\rvert.
    \end{equation}
    By the assumed e-property in $\LL_{bs}(S)$, the first term on the right-hand side of \eqref{last} vanishes. The second one vanishes as well, by the definition of $\mu$. As a~consequence, referring to Lemma~\ref{lem:weak-b-bs}, we get  $\text{w}\,\text{-}\lim_{k\to\infty}\delta_{x_{n_k}}P(t_{n_k}) = \mu$, which contradicts \eqref{eq:thm-e-prop-eq}.
\end{proof}
\vspace{5mm}

{\bf Acknowledgement.} Both authors acknowledge the support of the project ``Near-term Quantum Computers: challenges, optimal implementations and applications'' under grant number POIR.04.04.00-00-17C1/18-00, which is carried out within the Team-Net programme of the
Foundation for Polish Science, co-financed by the European Union under the European Regional Development Fund. 

Both authors express their gratitude to the anonymous Reviewer for their careful reading of our manuscript and their insightful comments and suggestions, which allowed us to significantly improve our work.

H.W.-\'S. would also like to extend special thanks to the Mathematical Institute at Leiden University, where part of this work was performed, for their hospitality. In particular, thanks are addressed to Dr. Sander C. Hille for all inspiring discussions.

Additional thanks are due to prof. dr hab. Łukasz Stettner and the participants of his seminar on \textit{Stochastic Processes} at the Institute of Mathematics of the Polish Academy of Sciences in Warsaw for generously sharing their comprehensive knowledge of the investigated subject after the lecture by H.W.-\'S., promoting the results of this paper.

\bibliographystyle{amsplain}

\end{document}